\documentclass[11pt]{amsart}

\usepackage{amsfonts}
\usepackage{amssymb}
\usepackage{graphicx}
\usepackage{amsmath,amsthm}
\usepackage{latexsym}
\usepackage{amscd}
\usepackage{xypic}
\usepackage{esint}
\usepackage{extarrows}
\usepackage{mathrsfs}
\usepackage{dsfont}

\usepackage[margin=1.5in]{geometry}

\usepackage{tikz} \usetikzlibrary{arrows,snakes,shapes,decorations.markings,patterns,calc}
\usepackage{pgfplots}
\usepackage{hyperref}

\allowdisplaybreaks

\numberwithin{equation}{section}

\newtheorem{prop}{Proposition}[section]
\newtheorem{thm}[prop]{Theorem}
\newtheorem{lemm}[prop]{Lemma}
\newtheorem{coro}[prop]{Corollary}

\newtheorem*{claim*}{Claim}

\theoremstyle{definition}
\newtheorem{defi}[prop]{Definition}

\newtheorem{rmk}[prop]{Remark}

\newcommand{\RR}{\mathbb{R}}

\newcommand{\cA}{\mathcal A}

\newcommand{\cF}{\mathcal F}

\renewcommand{\cH}{\mathcal H}

\def\fM{\mathfrak{M}}

\newcommand{\ep}{\epsilon}

\newcommand{\sing}{\text{sing}}
\newcommand{\reg}{\text{reg}}

\let\oldmarginpar\marginpar
\renewcommand\marginpar[1]{\-\oldmarginpar[\raggedleft\footnotesize #1]%
{\raggedright\footnotesize #1}}

\DeclareMathOperator{\dist}{dist}

\usepackage{graphicx}

\title[A Compactness Result for Energy-minimizing Harmonic Maps]{A Compactness Result for Energy-minimizing Harmonic Maps with Rough Domain Metric}
\author[D.~Cheng]{Da Rong Cheng}
\address{Department of Mathematics, Stanford University, Stanford, CA 94305}
\email{tjcheng@stanford.edu}

\begin{document}
\begin{abstract}
In 1996, Shi \cite{shi} generalized the $\ep$-regularity theorem of Schoen and Uhlenbeck \cite{su} to energy-minimizing harmonic maps from a domain equipped with a Riemannian metric of class $L^{\infty}$. In the present work we prove a compactness result for such energy-minimizing maps. As an application, we combine our result with Shi's theorem to give an improved bound on the Hausdorff dimension of the singular set, assuming that the map has bounded energy at all scales. This last assumption can be removed when the target manifold is simply-connected.
\end{abstract}

\maketitle



\section{Introduction and Statement of Main Results}
Let $N$ be a smooth compact Riemannian manifold, isometrically embedded in some Euclidean space $\RR^{m}$. Let $B$ denote the unit ball in $\RR^{n}$. The Sobolev space $W^{1, 2}(B; N)$ is defined by
\[
W^{1, 2}(B; N) = \{ u \in W^{1, 2}(B; \RR^{m})|\ u(x) \in N\text{ for a.e. }x \in B \}.
\]
Let $\cA$ denote the collection of open subsets of $B$. Given a positive number $\Lambda$, we let $\cF_{\Lambda}$ denote the class of funtionals $E: L^{2}(B; \RR^{m})\times \cA \to [0, +\infty]$ which have the form
\begin{equation}
\label{flambda}
E(u, A) = 
\left\{
\begin{array}{cl}
\int_{A}\sqrt{g}g^{ij}D_{j}u \cdot D_{i}u & ,\ u|_{A} \in W^{1, 2}(A; \RR^{m})\\
+\infty &,\ \text{otherwise}
\end{array},
\right.
\end{equation}
where $g$ is a Riemannian metric on $B$ of class $L^{\infty}$ satisfying
\begin{equation}
\label{ellipticity}
\Lambda^{-1}|\xi|^{2} \leq \sqrt{g}(x)g^{ij}(x)\xi_{i}\xi_{j} \leq \Lambda |\xi|^{2},\ \text{for a.e. }x \in B\text{ and all }\xi \in \RR^{n}.
\end{equation}
In the present work we are interested in the local minimizers of functionals in $\cF_{\Lambda}$. To be precise, let $\fM_{\Lambda}$ be the subset of $W^{1, 2}(B; N)$ consisting of maps $u$ for which there exists an $E \in \cF_{\Lambda}$ such that for each $B_{r}(x) \subset\subset B$ and each $v \in W^{1, 2}(B_{r}(x); N)$ with $u-  v \in W^{1, 2}_{0}(B_{r}(x); \RR^{m})$, we have
\begin{equation}
\label{mindef}
E(u ,B_{r}(x)) \leq E(v, B_{r}(x)).
\end{equation}
Finally we'll also denote the ordinary Dirichlet energy (with respect to the Euclidean metric) by
\[
E^{0}(u, A) = \int_{A}|du|^{2}.
\]
Our main result is the following.
\begin{thm}
\label{compactness}
Let $\{u_{k}\}$ be a sequence of maps in $\fM_{\Lambda}$ with 
\[
\sup\limits_{k}E^{0}(u_{k}, B_{r}(x)) < +\infty, \text{ for each }B_{r}(x) \subset\subset B.
\] 
Then, passing to a subsequence if necessary, there exists $u \in \fM_{\Lambda}$ such that 
\begin{enumerate}
\item $u_{k} \to u$ weakly in $W^{1, 2}(B_{r}(x); \RR^{m})$ and strongly in $L^{2}(B_{r}(x); \RR^{m})$ for each $B_{r}(x) \subset\subset B$.
\item Suppose $E_{k}$ and $E$ are functionals in $\cF_{\Lambda}$ that $u_{k}$ and $u$ locally minimize, respectively. Then for each $B_{r}(x) \subset\subset B$, we have
\begin{equation}
\label{Econv}
\lim\limits_{k \to \infty}E_{k}(u_{k}, B_{r}(x)) = E(u, B_{r}(x)).
\end{equation}
\end{enumerate}
\end{thm}

Before mentioning an application of Theorem \ref{compactness}, we recall that in 1996, Shi proved the following $\ep$-regularity result for maps in $\fM_{\Lambda}$.
\begin{thm}[\cite{shi}]
\label{shi}
There exists positive numbers $\ep,\ \tau$ and $\alpha$ depending only on $n$ and $\Lambda$ such that if $u \in \fM_{\Lambda}$ and $B_{r}(x) \subset\subset B$ satisfies
\[
r^{2-n}E^{0}(u, B_{r}(x)) \leq \ep.
\]
Then 
\[
(\tau r)^{2-n}E^{0}(u, B_{\tau r}(x)) \leq \frac{1}{2}r^{2-n}E^{0}(u, B_{r}(x)).
\]
\end{thm}
For a map $u \in \fM_{\Lambda}$, we define the regular set to be
\begin{equation}
\label{reg}
\reg\ u = \left\{ x \in B |\ u \text{ is H\"older continuous on a neighborhood of }x\right\}.
\end{equation}
The singular set is then defined to be the complement of $\reg\ u$:
\[
\sing\ u = B - \reg\ u.
\]
For technical reasons we also define 
\begin{equation}
\label{sing}
\sing_{E}\ u=  \left\{ x \in B |\ \liminf\limits_{r \to 0}r^{2-n}E^{0}(u, B_{r}(x)) \geq \frac{1}{2}(1 + \Lambda)^{-2}\epsilon \right\}.
\end{equation}
Then $\sing\ u \subset \sing_{E}\ u$ and Shi's theorem immediately implies that 
\begin{equation}
\label{n-2}
\cH^{n-2}(\sing_{E}\ u \cap B_{1/2}) = 0, \text{ for all }u \in \fM_{\Lambda}.
\end{equation}
Below we will combine Theorem \ref{shi} and Theorem \ref{compactness} to get an improved bound on the singular set of a map $u \in \fM_{\Lambda}$, assuming that its energy is bounded at all scales.
\begin{thm}
\label{n-2-e}
There exists $\epsilon$ depending only on $n,\ \Lambda$ and $E_{0}$ such that for all $u \in \fM_{\Lambda}$ satisfying
\begin{equation}
\label{bmo}
r^{2-n}E^{0}(u, B_{r}(x)) \leq E_{0}\text{ for all } x \in B_{1/2},\ r \in (0, 1/4),
\end{equation}
 we have $\cH^{n-2-\epsilon}(\sing\ u \cap B_{1/2}) = 0$.
\end{thm}
The strategy for proving Theorem \ref{n-2-e} is by contradiction: we first negate the statement to get a sequence of counterexamples. Then we rescale the sequence appropriately and use Theorem \ref{compactness} to pass to a limit map which violates Shi's theorem.

The assumption \eqref{bmo} is a rather strong one, and we don't know if it can be weakened in general. Nonetheless, in the case where $N$ is simply-connected, \eqref{bmo} can be removed thanks to a universal energy bound due to Hardt, Kinderlehrer and Lin.
\begin{thm}[\cite{hkl}]
\label{universal}
Assume $N$ is simply-connected. For each compact subset $K$ of $B$, there exists a constant $C = C(n, K, N, \Lambda)$ such that
\[
E^{0}(u, K) \leq C
\]
for any $u \in \fM_{\Lambda}$.
\end{thm}
Using this result we obtain the following corollary of Theorem \ref{n-2-e}.
\begin{coro}
\label{simplyconn}
Suppose $N$ is simply-connected, then there exists $\ep = \ep(\Lambda, N, n)$ such that $\cH^{n-2-\ep}(\sing\ u) = 0 $ for each $u \in \fM_{\Lambda}$.
\end{coro}

This paper is organized as follows: In section 2 we introduce the concept of $\Gamma$-convergence and state a compactness result for energy functionals in $\cF_{\Lambda}$. In section 3 we utilize $\Gamma$-convergence to prove Theorem \ref{compactness}. In section 4 we prove Theorem \ref{n-2-e}. Finally, in section 5 we show how to derive Corollary \ref{simplyconn} from Theorem \ref{n-2-e} and Theorem \ref{universal}.


\section{$\Gamma$-convergence and compactness of $\cF_{\Lambda}$}
In this section we introduce the concept of $\Gamma$-convergence, which will be integral to the proof of Theorem \ref{compactness}. For a general introduction to $\Gamma$-convergence, see \cite{dalmaso}.

\begin{defi}
Let $F_{k}: L^{2}(B; \RR^{m}) \to [0, +\infty]$ be functionals on $L^{2}(B; \RR^{m})$. For $w \in L^{2}(B; \RR^{m})$, we define
\begin{equation}
\Gamma-\limsup\limits_{k \to \infty}F_{k}(w) = \inf\{\limsup\limits_{k \to \infty}F_{k}(w_{k}) | w_{k} \to w \text{ in }L^{2}(B; \RR^{m})\}
\end{equation}
\begin{equation}
\Gamma-\liminf\limits_{k \to \infty}F_{k}(w) = \inf\{\liminf\limits_{k \to \infty}F_{k}(w_{k}) | w_{k} \to w \text{ in }L^{2}\}
\end{equation}
Moreover, we say that $\{F_{k}\}$ $\Gamma$-converges to $F$, denoted $F = \Gamma-\lim\limits_{k \to \infty}F_{k}$, if
\[
F  = \Gamma-\limsup\limits_{k\to \infty} F_{k} = \Gamma-\liminf\limits_{k\to\infty}F_{k}.
\]
\end{defi}
Since $L^{2}(B; \RR^{m})$ is a metric space, we have the following characterization for the $\Gamma$-limit.
\begin{lemm}[\cite{dalmaso}, Proposition 8.1]
\label{Gammacrit}
For $w \in L^{2}(B; \RR^{m})$, $F(w) = \Gamma-\limsup_{k \to \infty}F_{k}(w) = \Gamma-\liminf_{k\to\infty}F_{k}(w)$ if and only if
\begin{enumerate}
\item ($\liminf$-inequality) For each sequence $(w_{k})$ converging to $w$ in $L^{2}(B; \RR^{m})$, we have
\begin{equation}
\label{liminfineq}
F(w) \leq \liminf\limits_{k \to \infty}F_{k}(w_{k}).
\end{equation}

\item($\limsup$-inequality) There exists a sequence $(w_{k})$ converging to $w$ in $L^{2}(B; \RR^{m})$ such that
\begin{equation}
\label{limsupineq}
\limsup\limits_{k\to\infty}F_{k}(w_{k}) \leq F(w).
\end{equation}
\end{enumerate}
A sequence verifying \eqref{limsupineq} is called a \textbf{recovery sequence}.
\end{lemm}
The following compactness result is what makes $\Gamma$-convergence useful to us. 
\begin{prop}[\cite{dalmaso}, Theorem 22.2]
\label{flambdacompact}
Let $\{E_{k}\}$ be a sequence in $\cF_{\Lambda}$. Then, passing to a subsequence if necessary, there exists $E \in \cF_{\Lambda}$ such that
\[
E(\cdot, A) = \Gamma-\lim\limits_{k \to \infty}E_{k}(\cdot, A),\ \text{for all }A \in \cA.
\]
\end{prop}
\begin{rmk}
\label{scal}
In fact, \cite{dalmaso} considers only scalar-valued functions ($m=1$). Nonetheless, the vector-valued case ($m > 1$) follows quite easily. The key is that for each $E \in \cF_{\Lambda}$ and $u = (u^{1}, \dots, u^{m}) \in W^{1, 2}(B; \RR^{m})$, we can write
\[
E(u, A) = \sum\limits_{k=1}^{m}\int_{A}\sqrt{g}g^{ij}D_{j}u^{k}D_{i}u^{k} \equiv \sum\limits_{k=1}^{m}E_{\text{scal}}(u^{k}, A)
\]
and apply the case $m=1$ to $E_{\text{scal}}$.
\end{rmk}



\section{Compactness of $\fM_{\Lambda}$}
In this section we prove Theorem \ref{compactness}. Suppose $\{u_{k}\}$ is a sequence in $\fM_{\Lambda}$ with locally uniformly bounded energy, i.e.
\[
\sup\limits_{k}E^{0}(u_{k}, B_{r}(x))< +\infty, \text{ for each }B_{r}(x) \subset\subset B.
\] 
The first conclusion of the Theorem then follows by a standard diagonal argument, yielding a limit map $u \in W^{1, 2}(B; N)$. Now, for each $k$, suppose $E_{k} \in \cF_{\Lambda}$ is a functional minimized locally by $u_{k}$. By Proposition \ref{flambdacompact}, passing to a further subsequence is necessary, we may assume that there is $E \in \cF_{\Lambda}$ such that 
\begin{equation}
\label{Ekconvergence}
E(\cdot, A) = \Gamma-\lim\limits_{k \to \infty}E_{k}(\cdot, A),\text{ for each }A \in \cA.
\end{equation}
\begin{prop}
\label{umin}
$u$ minimizes $E$ locally. In particular, $u \in \fM_{\Lambda}.$
\end{prop}
\begin{proof}
It suffices to prove that for each $\theta \in (1/2, 1)$ and each $v \in W^{1, 2}(B_{\theta}; N)$ with $u - v \in W^{1, 2}_{0}(B_{\theta}; \RR^{m})$, we have 
\[
E(u, B_{\theta}) \leq E(v, B_{\theta}).
\]
Define $\tilde{v}$ by 
\[
\tilde{v} = 
\left\{
\begin{array}{cl}
v &, \text{ in }B_{\theta}\\
u &, \text{ in }B - B_{\theta}
\end{array}.
\right.
\]
Next we fix positive numbers $\delta$ and $\eta$, to be sent to zero later. By \eqref{Ekconvergence} and recalling Lemma \ref{Gammacrit}, there exists a sequence $\{v_{k}\}$ in $W^{1, 2}(B_{\theta(1 + \eta)}; \RR^{m})$ converging to $\tilde{v}$ strongly in $L^{2}(B_{\theta(1 + \eta)}; \RR^{m})$ such that
\begin{equation}
\label{vlimsup}
\limsup\limits_{k \to \infty}E_{k}(v_{k}, B_{\theta(1 + \eta)}) \leq E(\tilde{v}, B_{\theta(1 + \eta)}).
\end{equation}
Below we show that based on $\{v_{k}\}$ we can construct a new sequence $\{\tilde{v}_{k}\}$, still converging to $\tilde{v}$ strongly in $L^{2}(B_{\theta(1 + \eta)}; \RR^{m})$, such that \eqref{vlimsup} is preserved and that $\tilde{v}_{k}(x) \in N$ for a.e. $x$. This will be done in two steps.

\textbf{Step 1: Improve to $L^{\infty}$-convergence}

This construction is inspired by \cite{bd}.
Since $v_{k} \to \tilde{v}$ in measure, there exists a sequence $k_{p}$ going to infinity such that 
\begin{equation}
\label{measureineq}
\left| \left\{x \in B_{\theta(1 + \eta)} |\ |v^{h}_{k}(x) - \tilde{v}^{h}(x)| > \frac{1}{p}\right\} \right| < \frac{1}{p}, \text{ for all }k \geq k_{p}, h \in \{1, \cdots, m\}.
\end{equation}
Now we define a new sequence $\{w_{k} = (w^{1}_{k}, \dots, w^{m}_{k})\}$ as follows. Whenever $k$ satisfies $k_{p} \leq k < k_{p+1}$, we define
\begin{equation}
w^{h}_{k} = 
\left\{
\begin{array}{cl}
v^{h}_{k} &,\text{ if } |v^{h}_{k} - \tilde{v}^{h}| \leq \frac{1}{p}\\
\tilde{v}^{h} + \frac{1}{p} &, \text{ if } v^{h}_{k} > \tilde{v}^{h} + \frac{1}{p}\\
\tilde{v}^{h} - \frac{1}{p} &, \text{ if } v^{h}_{k} < \tilde{v}^{h} - \frac{1}{p}
\end{array}.
\right.
\end{equation}
Then $w_{k} \in W^{1, 2}(B_{\theta(1 + \eta)}; \RR^{m})$ and $\|w_{k}^{h} - \tilde{v}^{h}\|_{L^{\infty}} \leq \frac{1}{p}$ for $k_{p} \leq k < k_{p+1}$, so 
\begin{equation}
\label{linfconv}
\lim\limits_{k \to \infty} \|w_{k} - \tilde{v}\|_{L^{\infty}} = 0.
\end{equation}
Moreover, notice that by \eqref{measureineq} and the definition of $w_{k}$, we have, for $k_{p} \leq k < k_{p + 1}$
\begin{equation}
\label{vwmeas}
\left| \left\{ x\in B_{\theta(1 + \eta)}|\ w^{h}_{k}(x) \neq v^{h}_{k}(x)\right\} \right| \leq \frac{1}{p}.
\end{equation}
Moreover, recalling Remark \ref{scal} and again using the definition of $w_{k}$, we get
\begin{align*}
E_{k}(w_{k}, B_{\theta(1 + \eta}) &= \sum\limits_{h=1}^{m}\left( \int_{\{ w^{h}_{k} = v^{h}_{k} \}}\sqrt{g_{k}}g^{ij}_{k}D_{j}v^{h}_{k}D_{i}v^{h}_{k} + \int_{\{ w^{h}_{k} \neq v^{h}_{k} \}}\sqrt{g_{k}}g^{ij}_{k}D_{j}\tilde{v}^{h}D_{u}\tilde{v}^{h} \right)\\
&\leq E_{k}(v_{k}, B_{\theta(1 + \eta)}) + C(\Lambda)\sum\limits_{h=1}^{m}\int_{\{w_{k}^{h} \neq v_{k}^{h}\}}|d\tilde{v}^{h}|^{2}.
\end{align*}
Letting $k \to \infty$ and using \eqref{vwmeas}, \eqref{vlimsup}, we have
\begin{equation}
\label{wlimsup}
\limsup\limits_{k \to \infty}E_{k}(w_{k}, B_{\theta(1 + \eta)}) \leq E(\tilde{v}, B_{\theta(1 + \eta)}).
\end{equation}

\textbf{Step 2:  Projecting onto $N$}

Since $\tilde{v}(x) \in N$ a.e., by \eqref{linfconv}, we infer that 
\begin{equation}
\label{distconv}
d(w_{k}, N) \text{ converges to zero in }L^{\infty}(B_{\theta(1 + \eta)}). 
\end{equation}
Since $N$ is compact, we may assume that there exists $d > 0$
\[
N_{d} = \{x \in \RR^{m} |\ d(x, N) \leq d\}
\]
is strictly contained in a tubular neighborhood of $N$. Let $\pi$ denote the nearest-point projection onto $N$; then eventually $\tilde{w}_{k} = \pi\circ w_{k}$ is defined. By \eqref{wlimsup}, \eqref{distconv} and the smoothness of $\pi$ we infer that 
\[
\lim\limits_{k\to \infty}\|\tilde{w}_{k} - \tilde{v}\|_{L^{\infty}(B_{\theta(1 + \eta)}; \RR^{m})} = 0
\]
and that
\begin{equation}
\label{wtlimsup}
\limsup\limits_{k \to \infty}E_{k}(\tilde{w}_{k}, B_{\theta(1 + \eta)}) \leq E(\tilde{v}, B_{\theta(1 + \eta)}).
\end{equation}

To proceed, we need the following version of the Luckhaus lemma. A proof can be found in \cite{leon}
\begin{lemm}
\label{Luckhaus}
Let $N$ be a compact submanifold of $\RR^{m}$ with $N_{d}$ strictly contained in a tubular neighborhood of $N$. Let $L$ be a positive constant. Then there exists a constant $\delta(n, L, d)$ such that for all $\epsilon \in (0, \delta)$, if $u, v \in W^{1, 2}(B_{(1 + \epsilon)\rho}(y) - B_{\rho}(y); N)$ satisfy
\begin{equation*}
\rho^{2-n}\int_{B_{(1 + \epsilon)\rho}(y) - B_{\rho}(y)}|Du|^{2} + |Dv|^{2} \leq L,
\end{equation*} 
\begin{equation*}
\epsilon^{-2n}\rho^{-n}\int_{B_{(1 + \epsilon)\rho}(y) - B_{\rho}(y)}|u - v|^{2} \leq \delta^{2}.
\end{equation*}
Then there exists $w  \in W^{1, 2}(B_{(1 + \epsilon)\rho}(y) - B_{\rho}(y); N)$ such that $w= u$ near $\partial B_{\rho}(y)$, $w = v$ near $\partial B_{(1 + \epsilon)\rho}(y)$ and 
\begin{align*}
&\rho^{2-n}\int_{B_{(1 + \epsilon)\rho}(y) - B_{\rho}(y)}|Dw|^{2}\\ &\leq C\rho^{2-n}\int_{B_{(1 + \epsilon)\rho}(y) - B_{\rho}(y)}|Du|^{2} + |Dv|^{2} + C\epsilon^{-2}\rho^{-n}\int_{B_{(1 + \epsilon)\rho}(y) - B_{\rho}(y)}|u - v|^{2},
\end{align*}
where $C$ depends on $n, L$ and $\sup\limits_{p \in N_{d}} |(d\pi)_{p}|$.
\end{lemm}
To apply the lemma to our situation, we choose $M$ large enough so that
\begin{equation}
\label{unifbdd}
E^{0}(u_{k}, B_{\theta(1 + \eta)}) + E^{0}(\tilde{w}_{k}, B_{\theta(1 + \eta)}) < M\delta,\ \forall k.
\end{equation}
Consider the annuli
\[
A_{\eta, l} = B_{\theta(1 + l\frac{\eta}{M})} - B_{\theta(1 + (l-1)\frac{\eta}{M})},\ l = 1, 2, \dots, M.
\]
By \eqref{unifbdd}, there exists an $l$ such that
\begin{equation}
E^{0}(u_{k}, A_{\eta, l}) + E^{0}(\tilde{w}_{k}, A_{\eta, l})< \delta\text{ for infinitely many }k.
\end{equation}
Without loss of generality we assume that this is satisfied for all $k$.

Next let $\rho = \theta(1 + (l-1)\frac{\eta}{M})$ and $\epsilon = \frac{\eta}{2M}$.
Then we have 
\[
\theta(1 + (l-1)\frac{\eta}{M}) = \rho \leq (1 + \epsilon)\rho \leq \theta(1 + l\frac{\eta}{M}).
\]
Now since $\tilde{v} = u$ on $B - B_{\theta}$ and $\lim\limits_{k \to \infty} \|\tilde{v} - \tilde{w}_{k}\|_{L^{2}} = \lim\limits_{k \to \infty} \|v - u_{k}\|_{L^{2}} = 0$, we have
\begin{equation*}
\lim\limits_{k \to \infty}\int_{B_{(1 + \epsilon)\rho} - B_{\rho}}|\tilde{w}_{k} - u_{k}|^{2} = 0.
\end{equation*}
Moreover, it's clear that there is a constant $L$ such that
\[
\rho^{2-n}\int_{B_{(1 + \epsilon)\rho} - B_{\rho}}|D\tilde{w}_{k}|^{2} + |Du_{k}|^{2} \leq L \text{ for all }k,
\]
and as $k$ tends to infinity, eventually we have
\[
\epsilon^{-2n}\rho^{-n}\int_{B_{(1 + \epsilon)\rho} - B_{\rho}}|\tilde{w}_{k} - u_{k}|^{2} \leq \delta(n, L, d)^{2}.
\]
Thus we can apply Lemma \ref{Luckhaus} with $\tilde{w}_{k}, u_{k}$ in place of $u, v$, respectively, obtaining a sequence $\{s_{k}\}$ in $L^{2}(B_{(1 + \epsilon)\rho} - B_{\rho}; N)$ such that 
\begin{align*}
&\int_{B_{(1 + \epsilon)\rho} - B_{\rho}}|Ds_{k}|^{2}\\ &\leq C\int_{B_{(1 + \epsilon)\rho} - B_{\rho}}|D\tilde{w}_{k}|^{2} + |Du_{k}|^{2} + C\epsilon^{-2}\rho^{-2}\int_{B_{(1 + \epsilon)\rho}(y) - B_{\rho}(y)}|\tilde{w}_{k} - u_{k}|^{2}\\
&\leq C\delta + C\epsilon^{-2}\rho^{-2}o(1).
\end{align*}
Sending $k$ to $\infty$ in the above inequality, we get
\begin{equation}
\label{transitionbound}
\limsup\limits_{k \to \infty}E^{0}(s_{k}, B_{(1 + \epsilon)\rho} - B_{\rho}) \leq C\delta.
\end{equation}
Now we define 
\[
\tilde{v}_{k} = 
\left\{
\begin{array}{ll}
\tilde{w}_{k} & , \text{ on }B_{\rho}\\
s_{k} & ,\text{ on }B_{(1 +\epsilon)\rho} - B_{\rho}\\
u_{k} & ,\text{ on }B - B_{(1 + \epsilon)\rho}
\end{array}
\right.
.
\]
Since $\Gamma-\lim\limits_{k \to\infty}E_{k}(\cdot, A) = E(\cdot, A)$, for all open subset $A$ of $B$, by the $\liminf$-inequality \eqref{liminfineq}, we have
\begin{align*}
E(u, B_{\theta}) &\leq \liminf\limits_{k \to \infty} E_{k}(u_{k}, B_{\theta}) \leq \limsup\limits_{k \to \infty}E_{k}(u_{k}, B_{\theta})\\
& \leq \limsup\limits_{k \to \infty}E_{k}(u_{k}, B_{(1 + \epsilon)\rho})\\
& \leq \limsup\limits_{k \to \infty} E_{k}(\tilde{v}_{k}, B_{(1 + \epsilon)\rho})\ ( u_{k}\text{ is minimizing }) \\
& \leq \limsup\limits_{k \to \infty}\left( E_{k}(\tilde{w}_{k}, B_{\rho}) + CE^{0}(s_{k}, B_{(1 + \epsilon)\rho} - B_{\rho})\right)\\
&  \leq \limsup\limits_{k \to \infty} E_{k}(\tilde{w}_{k}, B_{\theta ( 1 + \eta)}) + C\delta\ (\text{ by }\eqref{transitionbound})\\
& \leq E(\tilde{v}, B_{\theta(1 + \eta)}) + C\delta\ (\text{ by }\eqref{wlimsup}).\\
\end{align*}
Since $\delta, \eta > 0$ is arbitrary,  we have 
\[
E (u, B_{\theta}) \leq E (\tilde{v}, B_{\theta}) =  E(v, B_{\theta}).
\]
This completes the proof of Proposition \ref{umin}.
\end{proof}

Next we prove the second conclusion of Theorem \ref{compactness}. For each $B_{\theta}(x)\subset\subset B$, we take the comparison map $v \in W^{1, 2}(B_{\theta}(x); N)$ in the previous proposition to be just $u$ itself restricted to $B_{\theta}(x)$. Then $\tilde{v}$ would just be $u$. Following the arguments of Proposition \ref{umin}, we have
\begin{align*}
E(u, B_{\theta}(x)) &\leq \liminf\limits_{k\to \infty}E_{k}(u_{k}, B_{\theta}(x))\\
&\leq \limsup\limits_{k \to\infty}E_{k}(u_{k}, B_{\theta}(x)) \leq E(u, B_{\theta(1 + \eta)}(x)) + C\delta,\ 
\end{align*}
for all $\delta, \eta >0$. From this the second conclusion of Theorem \ref{compactness} follows easily and we've completed the proof of the Theorem \ref{compactness}.


\section{Improved upper bound for the singular set dimension}
In this section we prove Theorem \ref{n-2-e}. As mentioned in the introduction, the proof is by contradiction. Therefore we fix $\Lambda$ and $E_{0}$ and suppose that there exists a sequence of maps $\{u_{k}\}$ in $\fM_{\Lambda}$ and a sequence of positive numbers $\epsilon_{k}$ converging to zero, such that each $u_{k}$ satisfies \eqref{bmo} and
\[
\cH^{n-2-\epsilon_{k}}(\sing\ u_{k}\cap B_{1/2}) > 0.
\]
Next, following \cite{su} we define, for any subset $A$ of $B$, 
\begin{equation}
\label{hinf}
\varphi^{s}(A) = \inf \{ \sum\limits_{i}r_{i}^{s}|\ A \subset \cup_{i}B_{r_{i}}(x_{i}) \}.
\end{equation}
Recall that (\cite{federer}, 2.10.2)
\begin{equation}
\label{hinfh}
\varphi^{s}(A) = 0 \text{ if and only if }\cH^{s}(A) = 0,
\end{equation}
and that (\cite{federer}, 2.10.19)
\[
\limsup\limits_{r \to 0}\frac{\varphi^{s}(A \cap B_{r}(x))}{r^{s}}\geq 2^{-s},\text{ for a.e.}\ x\in A.
\]
Hence for each $k$ we can choose $x_{k} \in \sing\ u_{k} \cap B_{1/2}$ and $r_{k} \in (0, 1/4)$ so that 
\begin{equation}
\label{density}
\frac{\varphi^{n-2-\epsilon_{k}}(\sing\ u_{k} \cap B_{r_{k}/2}(x_{k}))}{r_{k}^{n-2-\epsilon_{k}}}\geq c_{n}
\end{equation}
for some constant $c_{n}$ depending only on $n$.
Now define a sequence of rescaled maps by letting
\[
v_{k}(y) = u_{k}(x_{k} + r_{k}y),\ y \in B.
\]
Then \eqref{density} implies
\begin{equation}
\label{resdensity}
\varphi^{n-2-\ep_{k}}(\sing\ v_{k} \cap B_{1/2}) \geq c_{n}.
\end{equation}
Since each $u_{k}$ is in $\fM_{\Lambda}$, it is not hard to see that the sequence of rescaled maps $\{v_{k}\}$ is also in $\fM_{\Lambda}$. Moreover, by the bound \eqref{bmo}, we have
\[
\sup\limits_{k}E^{0}(v_{k}, B) = \sup\limits_{k}r_{k}^{2-n}E^{0}(u_{k}, B_{r_{k}}(x_{k})) \leq E_{0}.
\]Thus by Theorem \ref{compactness}, there exists $v \in \fM_{\Lambda}$ such that 
\begin{enumerate}
\item $v_{k} \to v$ stronlgly in $L^{2}(B_{\theta}(x); \RR^{m})$ for each $B_{\theta}(x) \subset\subset B$.
\item Suppose $v_{k}$ and $v$ minimize $E_{k}$ and $E$, respectively. Then for each $B_{\theta(x)} \subset\subset B$, \eqref{Econv} holds.
\end{enumerate}
Now for each covering $\{B_{r_{i}}(x_{i})\}$ of $\sing_{E}\ v \cap B_{1/2}$ by balls, let $K = B_{1/2} - \cup_{i}B_{r_{i}}(x_{i})$. Let $2d = \dist(K, \sing\ v \cap B_{1/2}) > 0$. Then by the definition of $\sing_{E}\ v$ (see \eqref{sing}), we can choose a finite covering $\{B_{s_{i}/2}(y_{i})\}_{i=1}^{Q}$ of $K$, with $y_{i} \in K$ and $s_{i} \leq d$ such that for each $i$,
\begin{equation}
\label{quantsingineq}
s_{i}^{2-n}E^{0}(v, B_{s_{i}}(y_{i})) \leq \frac{1}{2}(1 + \Lambda)^{-2}\ep.
\end{equation}
Now by \eqref{Econv}, there exists $k_{0}$ such that for all $k\geq k_{0}$ and $1 \leq i \leq Q$, 
\begin{equation}
\label{energyapproximation}
E_{k}(v_{k}, B_{s_{i}}(y_{i})) \leq E(v, B_{s_{i}}(y_{i})) + (2\Lambda)^{-1}s_{i}^{n-2}\ep.
\end{equation}
Hence for $k \geq k_{0}$ and $1 \leq i \leq Q$, 
\begin{align*}
E^{0}(v_{k}, B_{s_{i}}(y_{i})) &\leq \Lambda E_{k}(v_{k}, B_{s_{i}}(y_{i})) \text{ (by \eqref{ellipticity})}\\
& \leq \Lambda(E(v, B_{s_{i}}(y_{i})) + (2\Lambda)^{-1} s_{i}^{n-2}\ep) \text{ (by \eqref{energyapproximation})}\\
& < \Lambda( \Lambda E^{0}(v, B_{s_{i}}(y_{i})) + (2\Lambda)^{-1}s_{i}^{n-2}\ep) \text{ (again by \eqref{ellipticity})}\\
& \leq \Lambda ((2\Lambda)^{-1}\ep + (2\Lambda)^{-1}\ep)s_{i}^{n-2} \text{ (by \eqref{quantsingineq})}\\
&= s_{i}^{n-2}\ep.
\end{align*}
So by Theorem \ref{shi}, $B_{s_{i}/2}(y_{i}) \cap \sing\ v_{k} = \emptyset$. Hence for $k \geq k_{0}$, 
\[
\sing\ v_{k} \subset \cup_{i}B_{r_{i}}(x_{i}).
\]
Thus
\[
\sum\limits_{i}r_{i}^{n-2-\ep_{k}} \geq \varphi^{n-2-\ep_{k}}(\sing\ v_{k} \cap B_{1/2}) \geq c_{n},\ \text{for all }k\geq k_{0}.
\]
Letting $k$ tend to infinity, we get
\[
\sum\limits_{i}r_{i}^{n-2} \geq c_{n}.
\]
Since $\{B_{r_{i}}(x_{i})\}$ is an arbitrary covering of $\sing\ v \cap B_{1/2}$ by balls, we conclude from \eqref{hinf} that
\[
\varphi^{n-2}(\sing_{E}\ v\cap B_{1/2}) \geq c_{n} > 0.
\]
Hence by \eqref{hinfh}, this implies 
\[
\cH^{n-2}(\sing_{E}\ v\cap B_{1/2}) > 0,
\]
which is clearly in contradiction with \eqref{n-2}, and the proof of Theorem \ref{n-2-e} is complete.

\section{The case when $N$ is simply-connected}
In this section we assume in addition that $N$ is simply-connected and prove Corollary \ref{simplyconn}. By Theorem \ref{n-2-e} this reduces to verifying condition \eqref{bmo}. Given $u \in \fM_{\Lambda}$, let $E$ be a functional in $\cF_{\Lambda}$ of which $u$ is a local minimizer and suppose $E$ is given by \eqref{flambda} with some Riemannian metric $g$ of class $L^{\infty}$.

For each $B_{r}(x) \subset\subset B$ with $x \in B_{1/2}$ and $r \in (0, 1/4)$, we define 
\begin{equation}
u_{x, r}(y) = u(x + 2ry),\ y \in B.
\end{equation}
We also define $E_{x, r}: L^{2}(B; \RR^{m}) \times \cA \to [0, +\infty]$ by
\begin{equation}
E_{x ,r}(v, A) = 
\left\{
\begin{array}{cl}
\int_{A}\sqrt{g}(x + 2ry)g^{ij}(x + 2ry)D_{j}v(y) \cdot D_{i}v(y)dy & ,\ v|_{A} \in W^{1, 2}(A; \RR^{m})\\
+\infty &,\ \text{otherwise}
\end{array}
\right.
.
\end{equation}
Then it's not hard to see that $E_{x, r} \in \cF_{\Lambda}$ and $u_{x, r}$ is a local minimizer for $E_{x, r}$. Thus $u_{x, r} \in \fM_{\Lambda}$. Moreover, a straightforward computation shows that 
\begin{equation}
\label{energyrelation}
E^{0}(u, B_{r}(x)) = (2r)^{n-2}E^{0}(u_{x, r}, B_{1/2}).
\end{equation}
Now since $u_{x, r} \in \fM_{\Lambda}$, by Theorem \ref{universal} there is a constant $C = C(n, \Lambda, N)$ such that
\begin{equation}
\label{universalbound}
E^{0}(u_{x, r}, B_{1/2}) \leq C.
\end{equation}
Combining \eqref{energyrelation} and \eqref{universalbound}, we get
\[
r^{2-n}E^{0}(u, B_{r}(x)) \leq 2^{n-2}C.
\]
Hence condition \eqref{bmo} is verified with $E_{0} = 2^{n-2}C$ and Corollary \ref{simplyconn} follows immediately.

\bibliographystyle{amsalpha}
\bibliography{compactness}

\end{document}